\documentclass[reqno,twocolumn]{article}

\usepackage{graphicx}
\usepackage{amsmath,amssymb,amsfonts}
\usepackage{amsthm}
\usepackage{xcolor}
\usepackage{tikz}
\usepackage{tkz-fct}
\usepackage{mathtools}
\usepackage{fullpage}
\usepackage{hyperref}

\newtheorem{theorem}{Theorem}

\newtheorem{question}{Question}
\newtheorem{problem}{Problem}

\theoremstyle{definition}
\newtheorem{example}{Example}

\DeclareMathOperator{\sign}{sgn}

\usepackage{etoolbox}
\makeatletter
\patchcmd{\@maketitle}
  {\LARGE \@title}
  {\large \bfseries \MakeUppercase{\@title}}
  {}{}
\makeatother

\title{Error analysis of quantization combined with Hadamard transforms}

\author{
Matvei Kotov\\
\small V-Nova, London, UK\\
\small \texttt{matvei.kotov@v-nova.com}
\and
Lorenzo Ciccarelli\\
\small V-Nova, London, UK\\
\small \texttt{lorenzo.ciccarelli@v-nova.com}
}

\date{} % optional: remove date

\begin{document}

\maketitle

{\small
\noindent\textbf{Abstract:}
In this paper, we consider an image coding process consisting of the following four steps: a direct transformation, a direct quantization, an inverse quantization, and an inverse transformation, where Hadamard transforms are used for the transformation steps and a dead-zone quantizer is used for the quantization. The aim of this paper is to provide a theoretical tool for analyzing this process. We discuss error bounds for this process and bounds on the largest absolute value that the components of the result can attain. In order to obtain these bounds, we use methods of linear algebra and properties of Hadamard matrices. The obtained formulae depend on the size of the matrices, the parameters of the quantizer and the dequantizer, and a bound on the source values. Knowing the error bounds helps control the trade-off between compression efficiency and output quality. Knowing the bounds on the largest absolute value helps decide how many bits are needed to store the result. In addition, we demonstrate a connection between the norm~$\|\mathbf{H}\|_{\infty, 1}$ of a Hadamard matrix~$\mathbf{H}$ and the maximal excess~$\sigma([\mathbf{H}])$ of the equivalence class containing~$\mathbf{H}$.

\smallskip


\smallskip

}

\section{Introduction}
 In image and video compression, some lossy compression algorithms apply the following sequence of steps on a block-by-block basis: 
\begin{equation}\label{pipeline}
  \mathbf{x}' = \mathbf{IT}(\mathbf{IQ}(\mathbf{DQ}(\mathbf{DT}(\mathbf{x})))),
\end{equation}
where $\mathbf{x}$ is a block of the image written as a vector, $\mathbf{DT}$ is a direct transformation, $\mathbf{IT}$ is an inverse transformation, $\mathbf{DQ}$ is a direct quantization, and $\mathbf{IQ}$ is an inverse quantization. The transformation step is usually lossless, but is used to reorganize the signal for better quantization. The transform can be a discrete cosine transform, a discrete wavelet transform, a Hadamard transform, etc. 

The use of Hadamard matrices in image coding was first proposed by Pratt et al.~\cite{Pratt1969}. This approach was also studied by Kitajima et al.~\cite{KitajimaShimonoKurobe1980}. Philips and Denecker~\cite{PhilipsDenecker1997} considered a lossless version of the Hadamard transform. Hadamard transforms can be chosen because they can be computed fast using only additions, subtractions, and right shifts~\cite{Salomon2007}. 

Hadamard matrices are also used in video compression. For example, 
the Low Complexity Enhancement Video Coding standard, formally known as MPEG-5 Part 2 LCEVC (ISO/IEC 23094-2) and approved by ISO/IEC JTC 1/SC 29/WG04 (MPEG), employs Hadamard matrices of size~4 and 16~\cite{ISO230942,BattistaMeardiFerraraCiccarelliMaurerContiOrcioni2022,MeardiFerraraCiccarelliCobianchiPoularakisMaurerBattistaByagowi2020}.

Various quantization schemes are used in practice. For example, a mid-tread uniform quantizer is defined as
\begin{gather*}
DQ(x) = \left\lfloor\frac{x}\Delta + \tfrac12\right\rfloor\!,
\\
IQ(x) = \Delta x. 
\end{gather*}
A mid-riser uniform quantizer is defined as 
\begin{gather*}
DQ(x) =  \left\lfloor\frac{x}\Delta\right\rfloor\!,  \\ IQ(x) = \Delta\! \left(x + \tfrac{1}{2}\right)\!. 
\end{gather*}

Also, different rounding methods are used. For example, rounding towards zero can be used, i.e. negative numbers are rounded up and positive numbers are rounded down. For example, in this case, a mid-rised uniform quantizer is defined as
\begin{gather*}
DQ(x) =  \sign(x)\!\left\lfloor\frac{|x|}\Delta\right\rfloor\!, \\
 IQ(x) = \sign(x)  \Delta\!\left(|x| + \tfrac{1}{2}\right)\!. 
\end{gather*}

For more efficient compression, a dead-zone quantizer can be used: 
\begin{gather*}
DQ(x) = \sign(x)\max\!\left(0, \left\lfloor \frac{|x| - \delta}\Delta \right\rfloor +1\right)\!,
\\
IQ(x) = \sign(x)\!\left(\Delta\!\left(|x| -\tfrac{1}{2}\right) + \delta \right). 
\end{gather*}

Due to losslessness, the result $\mathbf{x}'$ in the sequence~(\ref{pipeline}) can be different from the source block $\mathbf{x}$. It is important to know bounds on the following two characteristics of this process: 

\begin{enumerate}
    \item the error made during the computation,
    \item the largest absolute value that the components of $\mathbf{x}'$ can attain.  
\end{enumerate}

Knowing the error bound helps control the trade-off between compression efficiency and output quality. Knowing the bounds on the largest absolute value helps decide how many bits are required to store $\mathbf{x}'$ to avoid overflow or to decide that clamping should be applied. 

Let $\mathbf{x} \in \mathbb{Z}^{n}$. In this paper, we consider the direct transformation and the inverse transformation steps based on Hadamard matrices. Since the decoder side should have as few operations as possible because it runs multiple times, the division by $n$ is performed during the direct transformation step:
$$
\mathbf{DT}(\mathbf{x}) = {\textstyle\frac{1}{n}} \mathbf{Hx}
\quad 
\text{and}
\quad
\mathbf{IT}(\mathbf{x}) = \mathbf{H}^{T}\mathbf{x},
$$
where $\mathbf{H}$ is a Hadamard matrix of size $n \times n$, and $\mathbf{x}$ is a block of $n$ pixels written as a vector. We will focus on Hadamard matrices generated by Sylvester's construction, but some of our results are true for any Hadamard matrix.

To make our result general enough, we use the following formulae for the direct quantization $\mathbf{DQ}(\mathbf{x}) = (DQ_1(x_1), \ldots, DQ_n(x_n))$ and the inverse quantization $\mathbf{IQ}(\mathbf{x}) = (IQ_1(x_1), \ldots, \allowbreak IQ_n(x_n))$:
\begin{gather*}
  DQ_i(x) = \sign(x) \!\left\lfloor \frac{\max(0, |x| + \delta_i)}{\Delta_i}\right\rfloor\!,
\\
  IQ_i(x) = \sign(x)(\Gamma_i |x| + \gamma_i),
\end{gather*}
where $\delta_i$ and $\gamma_i$ are integers, $\Delta_i$ and $\Gamma_i$ are positive integers, and $1 \leq i \leq n$.
Note that we allow $\Delta_i$ to be different from $\Gamma_i$.
In this paper, we consider $\delta_i \leq 0$, $\gamma_i \geq 0$, $-\delta_i \leq \Delta_i$ and $\gamma_i \leq \Gamma_i$. 

\begin{example}\label{IntroExample1}
  Let $\delta_i=\gamma_i=0$ and $\Delta_i=\Gamma_i=1000$.
  Therefore, $DQ_i(x) = \sign(x)\!\left\lfloor \frac{|x|}{1000} \right\rfloor$ and $IQ_i(x) = 1000 x$.
The vector $\mathbf{x}$ is defined below.
  Let $\mathbf{t}_1 = \mathbf{DT}(\mathbf{x})$, 
$\mathbf{t}_2 = \mathbf{DQ}(\mathbf{t}_1)$, and 
$\mathbf{t}_3 = \mathbf{IQ}(\mathbf{t}_2)$.
Then the vectors $\mathbf{x}$, $\mathbf{t}_1$, $\mathbf{t}_2$, $\mathbf{t}_3$, and $\mathbf{x}'$ are equal to 
  $$
    \begin{pmatrix}
      4016 \\4000 \\4000 \\4000 \\4000 \\-4000 \\4000 \\-4000 \\4000\\ 4000\\ -4000\\ -4000\\ 4000\\ -4000\\ -4000 \\4000
    \end{pmatrix}\!,
    \begin{pmatrix}
      1001 \\ 1001 \\ 1001 \\ -999 \\ 1001 \\ 1001 \\ 1001 \\ -999 \\ 1001 \\ 1001 \\ 1001 \\ -999 \\ -999 \\ -999 \\ -999 \\ 1001 
    \end{pmatrix}\!,
    \begin{pmatrix}
      1 \\ 1 \\1 \\0 \\1 \\1 \\1 \\0 \\1 \\1 \\1 \\0 \\0\\ 0\\ 0\\ 1
    \end{pmatrix}\!,
   \begin{pmatrix}
      1000 \\ 1000 \\ 1000 \\0 \\1000 \\1000 \\1000 \\0 \\1000 \\1000 \\1000 \\0 \\0\\ 0\\ 0\\ 1000
    \end{pmatrix}\!,\text{ and }
    \!\begin{pmatrix}
      10000 \\2000\\ 2000\\ -2000\\ 2000\\ 2000\\ 2000\\ -2000\\ 2000\\ 2000\\ 2000\\ -2000\\ -2000\\ -2000\\ -2000\\ 2000
    \end{pmatrix}\!,$$ respectively.
  
     Therefore, we can see that the first component became almost 2.5 times larger, and the error is almost 6 stepwidths.
\end{example}

As we mentioned earlier, it is important to know bounds on the error made during the computation and on the largest absolute value that the components of $\mathbf{x}'$ can attain. To formalize these characteristics, we will use the vector norm $\|\mathbf{x}\|_\infty = \max(|x_1|, \ldots, |x_n|)$. This brings us to the following mathematical problems.
\begin{problem}\label{Problem1}
  Find a function $f$ such that $\|\mathbf{x}' - \mathbf{x}\|_\infty \leq  f(\|\mathbf{x}\|_\infty, n, \boldsymbol{\Delta}, \boldsymbol{\Gamma}, \boldsymbol{\delta}, \boldsymbol{\gamma})$.
\end{problem}

\begin{problem}\label{Problem2}
  Find a function $g$ such that $\|\mathbf{x}'\|_\infty \leq g(\|\mathbf{x}\|_\infty, n, \boldsymbol{\Delta}, \boldsymbol{\Gamma}, \boldsymbol{\delta}, \boldsymbol{\gamma})$.
\end{problem}

Let us demonstrate how these bounds can be used.
If each component of $\mathbf{x}$ is stored using $k$ bits and these values are signed, i.e., $ -2^{k-1} \leq x_{i} \leq 2^{k-1} - 1$, then we can write $\|\mathbf{x}\|_\infty \leq 2^{k-1}$.
For example, if $k = 8$, then $\|\mathbf{x}\|_\infty \leq 128$; if $k = 16$, then $\|\mathbf{x}\|_\infty \leq 32768$, etc.
If, while solving Problem~\ref{Problem2}, we obtained an inequality $\|\mathbf{x}'\|_\infty \leq C  \|\mathbf{x}\|_\infty$, where $C$ is a constant, then it means
$$
-C  2^{k-1} \leq x'_i \leq C 2^{k-1}.
$$
To find the number of bits required to store $x'_i$, we need to find an integer $K$ such that
$$
-2^{K -1} \leq -C  2^{k-1} \leq x'_i \leq C  2^{k-1} \leq 2^{K - 1} - 1.
$$
Solving the inequalities, we obtain
\begin{equation}\label{FormInro}
  K = \lceil \log_2(C 2^{k - 1} + 1)\rceil + 1.
\end{equation}
For example, suppose that we have $\|\mathbf{x}'\|_\infty \leq 1.5 \|\mathbf{x}\|_\infty$ and $\|\mathbf{x}\|_\infty \leq 32768$.
Hence, for each $i$, $-49152 \leq x'_{i} \leq 49152$.
From Formula~(\ref{FormInro}), we obtain $K = 17$.
Thus, we must either use $17$ bits for each component or clamp the components of $\mathbf{x}'$ to be able to store them using $16$ bits.

In this paper, we solve Problem~\ref{Problem1} and Problem~\ref{Problem2}.
To solve the problems, we use methods of linear algebra and properties of Hadamard matrices.
Our solution to Problem~\ref{Problem1} is given in Theorem~\ref{Theorem1} and our solution to Problem~\ref{Problem2} is given in Theorem~\ref{Theorem2}.

Note that error bounds for the fast Fourier transform have been widely studied; see, for example,~\cite{Brisebarre2018,TascheZeuner2001,Ramos1971}.
The theory of quantization noise in digital signal processing is considered in~\cite{WidrowKollar2008}.

The remainder of this paper is structured into five parts.
The next section recalls the definitions of vector and matrix norms and some results about these norms and about Hadamard matrices.
Section~\ref{ErrBoundSection} is devoted to $\|\mathbf{x} - \mathbf{x}'\|_\infty$: it contains some examples to demonstrate the effect of quantization and rounding and shows how methods of linear algebra can be applied to obtain error bounds.
Additionally, we discuss asymptotic properties of these bounds and the relative error.
In Section~\ref{LargestSection}, we consider bounds on the largest value of~$\|\mathbf{x}'\|_\infty$. It turned out that the matrix norm $\|\cdot\|_{\infty, 1}$ is very useful in solving Problem~\ref{Problem2}. In Section~\ref{ConnectionSection}, we demonstrate the connection between the norm $\|\mathbf{H}\|_{\infty, 1}$ of a Hadamard matrix~$\mathbf{H}$ and the maximal excess $\sigma([\mathbf{H}])$ of the equivalence class containing $\mathbf{H}$.
The conclusion summarizes our findings and also contains some open questions.

\section{Preliminaries}

In this section, we recall the definitions of vector and matrix norms and the definition of Hadamard matrices.
Also, we discuss some properties of Hadamard matrices.

A \textit{Hadamard matrix} is a square matrix whose entries are either $+1$ or $-1$ and whose rows are mutually orthogonal.
Let $\mathcal{H}_n$ denote the set of all Hadamard matrices of order $n$.

If $\mathbf{H} \in\mathcal{H}_n $, then
\begin{equation}\label{HadamardProperty}
  \mathbf{H}^T \mathbf{H} = n\mathbf{I}_n,
\end{equation}
where $\mathbf{I}_n$ is the $n \times n$ identity matrix.

If $\mathbf{H}$ is a Hadamard matrix, then it can be proved that its order must be $1$, $2$, or a multiple of $4$.
The existence of a Hadamard matrix of order $4k$ for every positive integer $k$ remains an open question, known as the Hadamard conjecture.

The Kronecker product provides a way to construct Hadamard matrices.
If $\mathbf{H}_1 \in \mathcal{H}_m$ and $\mathbf{H}_2 \in \mathcal{H}_n$, then $\mathbf{H}_1 \otimes \mathbf{H}_2 \in \mathcal{H}_{mn}$. 
Sylvester's construction provides a way to construct Hadamard matrices of order $2^k$:
\begin{equation}\label{Sylvester}
\mathbf{H}_1 = \!\begin{pmatrix}1\end{pmatrix}\!, \mathbf{H}_2 = \!\begin{pmatrix}1 & 1 \\ 1 & - 1\end{pmatrix}\!, 
  \mathbf{H}_{2^k} = \mathbf{H}_2 \otimes \mathbf{H}_{2^{k-1}}.
\end{equation}

We refer the reader to~\cite{Horadam2007} for more information on Hadamard matrices and their applications in image compression.

Let $\mathbf{x} \in \mathbb{R}^n$.
We will use the following standard notation from linear algebra:
\begin{gather}
  \nonumber \|\mathbf{x}\|_1 = |x_1| + \ldots + |x_n|,
\\
\label{VecNorm2}
\|\mathbf{x}\|_2 = \sqrt{x_1^2 + \ldots + x_n^2 }, 
\\
\nonumber  \|\mathbf{x}\|_\infty = \max(|x_1|, \ldots, |x_n|).
\end{gather}

We recall the following well-known inequalities:
\begin{equation}\label{Norm2vsNorm1}
  \|\mathbf{x}\|_\infty \leq \|\mathbf{x}\|_2 \leq \|\mathbf{x}\|_1 \leq \sqrt{n} \|\mathbf{x}\|_2 \leq n \|\mathbf{x}\|_\infty.
\end{equation}

%For any norm, we, of course, have
%$$
%  \|\mathbf{x} + \mathbf{y}\|_p \leq \|\mathbf{x}\|_p + \|\mathbf{y}\|_p \quad \text{and}\quad \|a %\mathbf{x}\|_p = |a|\|\mathbf{x}\|_p.
%$$

Let $\|\cdot \|_p$ and $\|\cdot \|_q$ be vector norms.
The matrix norm induced by these norms is defined as
$$
  \|\mathbf{A}\|_{p, q} = \sup_{\mathbf{x} \neq 0} \frac{\|\mathbf{A}\mathbf{x}\|_q}{\|\mathbf{x}\|_p} .
$$
If the vector norms are the same, we write $\|\mathbf{A}\|_{p}$ instead of $\|\mathbf{A}\|_{p, p}$.
Let $\mathbf{A}$ and $\mathbf{B}$ be matrices of size $n \times n$.
The following inequality hold:
\begin{gather}
\label{NormIneq}
  \|\mathbf{A}\mathbf{x}\|_q \leq \|\mathbf{A}\|_{p, q}  \|\mathbf{x}\|_p,
\shortintertext{It is known that} 
\label{Norm1}
  \|\mathbf{A}\|_1 = \max_{1 \leq j \leq n}{\sum_{i = 1}^n{|a_{ij}|}},
\\
\label{NormInfty}
  \|\mathbf{A}\|_\infty = \max_{1 \leq i \leq n}{\sum_{j = 1}^n{|a_{ij}|}},
\\
\label{Norm2}
  \|\mathbf{A}\|_2 = \sqrt{\lambda_{\max}(\mathbf{A}^T\mathbf{A})},
\end{gather}
where $\lambda_{\max}(\mathbf{A}^T\mathbf{A})$ is the largest eigenvalue of the matrix $\mathbf{A}^T\mathbf{A}$.

\begin{example}\label{Ex1}
  Consider a Hadamard matrix $\mathbf{H}$ of order $n$.
  Since all the elements of $\mathbf{H}$ are either $+1$ or $-1$, using~(\ref{Norm1}) and~(\ref{NormInfty}), it is easy to see that 
  $$
    \|\mathbf{H}\|_1 = \|\mathbf{H}\|_\infty = \|\mathbf{H}^T\|_1 = \|\mathbf{H}^T\|_\infty = n.
  $$
  Also, using~(\ref{HadamardProperty}) and~(\ref{Norm2}), we obtain 
  \begin{equation}\label{HnNorm2}
    \|\mathbf{H}\|_2 = \sqrt{n}.
  \end{equation}
\end{example}

It was proved in~\cite{Rohn2000} that 
\begin{gather}
\label{NormInfty1}
  \|\mathbf{A}\|_{\infty, 1} = \max\{\|\mathbf{A}\mathbf{x}\|_1 \mid \mathbf{x} \in \{-1, 1\}^n\},
\\
\nonumber
\|\mathbf{A}\|_{1, \infty} = \max_{i, j}|a_{ij}|.
\end{gather}

\begin{example}\label{Ex2}
  Consider a Hadamard matrix $\mathbf{H}$ of order $n$. 
  Using~(\ref{Norm2vsNorm1}), (\ref{NormIneq}), and (\ref{HnNorm2}), we obtain the following chain of inequalities:
  \begin{multline*}
    \|\mathbf{H} \mathbf{x}\|_1 \leq \sqrt{n} \|\mathbf{H} \mathbf{x}\|_2 \leq
    \sqrt{n}   \|\mathbf{H}\|_2  \|\mathbf{x}\|_2 =
     n \|\mathbf{x}\|_2.
  \end{multline*}
  Therefore, from (\ref{VecNorm2}) and (\ref{NormInfty1}), it follows that
  \begin{multline}\label{Formula2}
    \|\mathbf{H}\|_{\infty, 1} =
    \max\{\|\mathbf{H}\mathbf{x}\|_1 \mid \mathbf{x} \in \{-1, 1\}^n\} \\ \leq 
    \max\{n \|\mathbf{x}\|_2 \mid \mathbf{x} \in \{-1, 1\}^n\} \\= 
    n  \sqrt{1 + \ldots + 1} = n^{\frac32}.
  \end{multline}
\end{example}

\begin{example}\label{Ex3}
  Consider Sylvester's construction~(\ref{Sylvester}).
  We have the following chain of equalities:
  \begin{multline*}
    \left\|\mathbf{H}_{2^{k+2}}\!\begin{pmatrix}\mathbf{x} \\ \mathbf{x} \\ -\mathbf{x} \\ \mathbf{x}\end{pmatrix}\!\right\|_1 
    \\=
    \left\|\!
    \begin{pmatrix}
      \mathbf{H}_{2^k} & \mathbf{H}_{2^k} & \mathbf{H}_{2^k} & \mathbf{H}_{2^k} \\
      \mathbf{H}_{2^k} & -\mathbf{H}_{2^k} & \mathbf{H}_{2^k} & -\mathbf{H}_{2^k} \\
      \mathbf{H}_{2^k} & \mathbf{H}_{2^k} & -\mathbf{H}_{2^k} & -\mathbf{H}_{2^k} \\
      \mathbf{H}_{2^k} & -\mathbf{H}_{2^k} & -\mathbf{H}_{2^k} & \mathbf{H}_{2^k}
    \end{pmatrix}\!
    \!\begin{pmatrix}
      \mathbf{x} \\ 
      \mathbf{x} \\ 
      -\mathbf{x} \\ 
      \mathbf{x}
    \end{pmatrix}\!
    \right\|_1
    \\=
    \left\|\!\begin{pmatrix}
      \mathbf{H}_{2^k}\mathbf{x} + \mathbf{H}_{2^k}\mathbf{x} -\mathbf{H}_{2^k}\mathbf{x} + \mathbf{H}_{2^k}\mathbf{x} \\
      \mathbf{H}_{2^k}\mathbf{x} - \mathbf{H}_{2^k}\mathbf{x} - \mathbf{H}_{2^k}\mathbf{x} - \mathbf{H}_{2^k}\mathbf{x} \\
      \mathbf{H}_{2^k}\mathbf{x} + \mathbf{H}_{2^k}\mathbf{x} + \mathbf{H}_{2^k}\mathbf{x} - \mathbf{H}_{2^k}\mathbf{x} \\
      \mathbf{H}_{2^k}\mathbf{x} - \mathbf{H}_{2^k}\mathbf{x} + \mathbf{H}_{2^k}\mathbf{x} + \mathbf{H}_{2^k}\mathbf{x}
    \end{pmatrix}\!\right\|_1
    \\= 
    \left\|\!\begin{pmatrix}
      2\mathbf{H}_{2^k}\mathbf{x} \\
      -2\mathbf{H}_{2^k}\mathbf{x} \\
      2\mathbf{H}_{2^k}\mathbf{x} \\
      2\mathbf{H}_{2^k}\mathbf{x}
    \end{pmatrix}\!\right\|_1
%    \\=
%    \|2\mathbf{H}_{2^k}\mathbf{x}\|_1 +
%    \|{-2}\mathbf{H}_{2^k}\mathbf{x}\|_1 +
%    \|2\mathbf{H}_{2^k}\mathbf{x}\|_1 +
%    \|2\mathbf{H}_{2^k}\mathbf{x}\|_1
    \\= 
    2\|\mathbf{H}_{2^k}\mathbf{x}\|_1 +
    2\|\mathbf{H}_{2^k}\mathbf{x}\|_1 +
    2\|\mathbf{H}_{2^k}\mathbf{x}\|_1 +
    2\|\mathbf{H}_{2^k}\mathbf{x}\|_1
    \\= 
    8\|\mathbf{H}_{2^k}\mathbf{x}\|_1.
  \end{multline*}

  Therefore,
  \begin{multline}\label{Formula1}
    \|\mathbf{H}_{2^{k+2}}\|_{\infty, 1} \\= 
    \max\left\{\|\mathbf{H}_{2^{k+2}}\mathbf{x}\|_1 \mid \mathbf{x} \in \{-1, 1\}^{2^{k+2}}\right\}\\ \geq 
    \max\left\{\left\|\mathbf{H}_{2^{k+2}}
    \!\begin{pmatrix}
      \mathbf{x} \\ 
      \mathbf{x} \\ 
      -\mathbf{x} \\ 
      \mathbf{x}
    \end{pmatrix}\!
    \right\|_1  : \mathbf{x} \in \{-1, 1\}^{2^{k}}\right\} 
  \\  = 
    \max\left\{ 8\|\mathbf{H}_{2^k}\mathbf{x}\|_1   : \mathbf{x} \in \{-1, 1\}^{2^{k}}\right\} 
\\    = 
%    8 \max\left\{\|\mathbf{H}_{2^k}\mathbf{x}\|_1  :  \mathbf{x} \in \{-1, 1\}^{2^{k}}\right\} 
%\\    = 
    8 \|\mathbf{H}_{2^k}\|_{\infty, 1}.
  \end{multline}

  Using a computer, we can compute $\|\mathbf{H}_{1}\|_{\infty, 1} = 1$, $\|\mathbf{H}_{2}\|_{\infty, 1} = 2$, $\|\mathbf{H}_{4}\|_{\infty, 1} = 8$, $\|\mathbf{H}_{8}\|_{\infty, 1} = 20$, $\|\mathbf{H}_{16}\|_{\infty, 1} = 64$, and $\|\mathbf{H}_{32}\|_{\infty, 1} = 160$. 
  Using this fact and~the inequality (\ref{Formula1}), for even $k$, we have
  \begin{equation}\label{Formula4}
    \|\mathbf{H}_{2^{k}}\|_{\infty, 1} \geq 8^\frac{k}{2},
  \end{equation}
  and for odd $k \geq 3$, we have
  \begin{equation}\label{Formula5}
    \|\mathbf{H}_{2^{k}}\|_{\infty, 1} \geq {\textstyle \frac{5\sqrt{2}}8}  8^\frac{k}{2}.
  \end{equation}
  From Formula~(\ref{Formula2}), it follows that
  \begin{equation}\label{Formula3}
    \|\mathbf{H}_{2^k}\|_{\infty, 1} \leq (2^{k})^{\frac32} = 8^{\frac{k}{2}}.
  \end{equation}
  Combining Formulae~(\ref{Formula4}) and (\ref{Formula3}), for even $k$, we obtain the identity
  \begin{equation}\label{ExactNormEven}
    \|\mathbf{H}_{2^{k}}\|_{\infty, 1} = 8^\frac{k}2.
  \end{equation}
  Combining Formulae~(\ref{Formula5}) and (\ref{Formula3}), for odd $k \geq 3$, we obtain the double inequality
  $$
    {\textstyle \frac{5\sqrt{2}}8} 8^\frac{k}2 \leq  \|\mathbf{H}_{2^{k}}\|_{\infty, 1} \leq 8^\frac{k}2.
  $$
 
  Other ways of computing $\|\cdot\|_{\infty, 1}$ are discussed in Section~\ref{ConnectionSection}.
\end{example}

The norm $\|\mathbf{A}\|_{\infty, 1}$ and Formula~(\ref{NormIneq}) will help us find a bound on the number of large components of $\mathbf{A}\mathbf{x}$. 

\begin{example}
  Consider the matrix $\mathbf{H}_{16}$. 
  From~(\ref{ExactNormEven}), we have $\|\mathbf{H}_{16}\|_{\infty, 1} = 64$.
  Let $\mathbf{x} \in \mathbb{R}^{16}$ and $\|\mathbf{x}\|_\infty \leq 1000$.
  In this case, it is impossible to have seven components of the vector $\mathbf{H}_{16}\mathbf{x}$ equal to 10000. 
  Indeed, $7 \cdot 10000 = 70000$, but the sum of the absolute values of all the components $\|\mathbf{H}_{16}\mathbf{x}\|_1$ should be less than or equal to $\|\mathbf{H}_{16}\|_{\infty, 1} \|\mathbf{x}\|_\infty = 64 \cdot 1000 = 64000$.
\end{example}

\section{Error bounds}\label{ErrBoundSection}

In this section, we discuss upper bounds on the error $\|\mathbf{x} - \mathbf{x}'\|_\infty$.
To get started, we consider a simple case: $\mathbf{DQ}(\mathbf{x}) = \mathbf{RZ}(\mathbf{x})$
and 
$\mathbf{IQ}(\mathbf{x}) = \mathbf{x}$, where  $\mathbf{RZ}(\mathbf{x})$ is rounding
towards zero of each element of $\mathbf{x}$.
Next, we consider the case $\Delta_i = \Gamma_i = C$, $1 \leq i \leq n$.
And finally, we investigate the general case.
After that, we consider the asymptotic properties of the obtained bound and the relative error.
The main result of this section is Theorem~\ref{Theorem1}.

Note that $\mathbf{IT}(\mathbf{DT}(\mathbf{x})) = \mathbf{x}$ for any $\mathbf{x}$.
Also, note that $\mathbf{DT}$ and $\mathbf{IT}$ are linear.
%$$
%  \mathbf{DT}(a\mathbf{x} + b\mathbf{y}) = a\mathbf{DT}(\mathbf{x}) + b\mathbf{DT}(\mathbf{y})$$ and $$
%  \mathbf{IT}(a\mathbf{x} + b\mathbf{y}) = a\mathbf{IT}(\mathbf{x}) + b\mathbf{IT}(\mathbf{y}).
%$$
 
Consider the case $\mathbf{DQ}(\mathbf{x}) = \mathbf{RZ}(\mathbf{x})$
and $\mathbf{IQ}(\mathbf{x}) = \mathbf{x}$, i.e.~$\delta_i = \gamma_i = 0$ and $\Delta_i = \Gamma_i = 1$.
In this case, Formula~(\ref{pipeline}) can be rewritten as
$
  \mathbf{x}' = \mathbf{H}^{T}( \mathbf{RZ} ({\textstyle{\frac{1}{n}}}\mathbf{H} \mathbf{x})).
$
The nonlinear part of the pipeline is the transformation $\mathbf{RZ}(\mathbf{x})$.

\begin{example}\label{Ex11}
  Sometimes, the magnitude of numbers can become larger after applying the direct and inverse transforms, despite the fact that rounding makes the magnitude of numbers smaller or does not change.
    Let $\mathbf{t}_1 = \frac{1}{16}\mathbf{H}_{16}(\mathbf{x})$,
$\mathbf{t}_2 =\mathbf{RZ} (\mathbf{t_1})$, and the vector $\mathbf{x}$ be defined below.
Then the vectors $\mathbf{x}$, $\mathbf{t}_1$, $\mathbf{t}_2$, and $\mathbf{x}'$ are equal to 
  $$
    \begin{pmatrix}
      55 \\ -5 \\ -5 \\ -5 \\ -4096 \\ -5 \\ -5 \\ -4096 \\ -5 \\ -4096 \\ -5 \\ -4 \\ -5 \\ -2 \\ -5 \\ -4
    \end{pmatrix}\!,
    \begin{pmatrix}
      -768 \\ -251.875 \\ -252.25 \\ 259.375 \\ 259.125 \\ -252 \\ -251.625 \\ -763.25 \\ 259.25 \\ -252.125 \\ 771 \\ 259.625 \\ 259.625 \\ 771 \\ -252.125 \\ 259.25
    \end{pmatrix}\!,
    \begin{pmatrix}
      -768 \\ -251 \\ -252 \\ 259 \\ 259 \\ -252 \\ -251 \\ -763 \\ 259 \\ -252 \\ 771 \\ 259 \\ 259 \\ 771 \\ -252 \\ 259
    \end{pmatrix}\!, 
    \begin{pmatrix}
      55 \\ -5 \\ -5 \\ -5 \\ -4093 \\ -5 \\ -5 \\ -4097 \\ -5 \\ -4093 \\ -9 \\ -5\\ -5 \\ -1 \\ -5 \\ -5
    \end{pmatrix}\!,
  $$respectively.
  
  Therefore, $\|\mathbf{x}\|_{\infty} = 4096$ and $\|\mathbf{x}'\|_{\infty} = 4097$.
\end{example}

Let $\mathbf{y} = \frac{1}{n}\mathbf{H}\mathbf{x}$.
We can write $\mathbf{RZ}(\mathbf{y}) = \mathbf{y} + \mathcal{E}(\mathbf{y})$, where $\| \mathcal{E}(\mathbf{y})\|_\infty < 1$, because $|y - RZ(y)| < 1$ for any number.
Thus, 
%\begin{multline*}
$  \mathbf{x}' = \mathbf{IT}(\mathbf{RZ}(\mathbf{DT}(\mathbf{x}))) = \mathbf{IT}(\mathbf{y} + \mathcal{E}(\mathbf{y})) = \mathbf{H}^T({\textstyle \frac{1}{n}}\mathbf{H}\mathbf{x}) + \mathbf{H}^T(\mathcal{E}(\mathbf{y})) =  \mathbf{x} + \mathbf{H}^T(\mathcal{E}(\mathbf{y})).
%\end{multline*}
$
Therefore, using Example~\ref{Ex1}, we have
$
%\begin{multline*}
  \|\mathbf{x}' - \mathbf{x}\|_\infty = \|\mathbf{x} + \mathbf{H}^T(\mathcal{E}(\mathbf{y})) - \mathbf{x}\|_\infty  \leq \|\mathbf{H}^T(\mathcal{E}(\mathbf{y}))\|_\infty
  \leq \|\mathbf{H}^T\|_\infty \|\mathcal{E}(\mathbf{y})\|_\infty  = n.
%\end{multline*}
$
Hence, the maximal difference is bounded by $n$.
Note that this bound does not depend on $\mathbf{x}$.

\begin{example}
  For example, if $n = 16$ and $\|\mathbf{x}\|_{\infty} \leq 4096$, then $\|\mathbf{x}'\|_{\infty} \leq 4096 + 16 = 4112$.
\end{example}

Now, consider the case with nontrivial quantization.
Let again $\mathbf{y} = \frac{1}{n}\mathbf{H}\mathbf{x}$.
The nonlinear part of~(\ref{pipeline}) is $\mathbf{IQ}(\mathbf{DQ}(\mathbf{y}))$.
We can write $\mathbf{IQ}(\mathbf{DQ}(\mathbf{y})) = \mathbf{y} + \mathcal{E}(\mathbf{y})$, where $\mathcal{E}(\mathbf{y})$ is the result of the quantization errors.
We have 
$
%\begin{multline*}
  \mathbf{x}' 
  = \mathbf{IT}(\mathbf{IQ}(\mathbf{DQ}(\mathbf{DT}(\mathbf{x})))) 
  = \mathbf{IT}({\textstyle \frac{1}{n}}\mathbf{H}\mathbf{x} + \mathcal{E}(\mathbf{y}))
  = \mathbf{x} + \mathbf{IT}(\mathcal{E}(\mathbf{y})).
%\end{multline*}
$
Therefore, 
%repeating the reasoning from the previous case,
we have
\begin{multline}\label{Est1}
  \|\mathbf{x}' - \mathbf{x}\|_\infty 
  \leq \|\mathbf{H}^T(\mathcal{E}(\mathbf{y}))\|_\infty
  \leq \|\mathbf{H}^T\|_\infty \|\mathcal{E}(\mathbf{y})\|_\infty
  \\ = n \|\mathcal{E}(\mathbf{y})\|_\infty.
\end{multline}

Let us find a bound on $\|\mathcal{E}(\mathbf{y})\|_\infty$.
The graph of $z(y) = IQ_i(DQ_i(y))$ looks like a staircase; see Figure~\ref{Fig1}.
Note that $IQ_i(DQ_i(y))$ can be rewritten in the following way:
$$
IQ_i(DQ_i(y)) =\!\left\{
\begin{array}{ll}
-\!\left\lfloor\frac{\delta_i - y}{\Delta_i}\right\rfloor\!\Gamma_i - \gamma_i& \text{if } y \leq \delta_i -\Delta_i,\\
\left\lfloor\frac{y + \delta_i}{\Delta_i}\right\rfloor\!\Gamma_i + \gamma_i& \text{if } y \geq \Delta_i - \delta_i, \\
0& \text{otherwise}.
\end{array}
\right.
$$

\begin{figure}
  \centering
  \begin{tikzpicture}[scale=0.6]
    \tkzInit[xmin = -55, xmax = 55, ymin = -55, ymax = 55, xstep=10, ystep=10]
    \tkzAxeXY[/tkzdrawX/label=$y$,/tkzdrawY/label=$z$]
    \draw[thick] (-1.45, 0) -- (1.45, 0);
   \node [circle, draw, fill=none, line width = .5pt, inner sep = 0pt, minimum size = 3pt] (ca) at (-1.5, 0) {};
   \node [circle, draw, fill=none, line width = .5pt, inner sep = 0pt, minimum size = 3pt] (ca) at (1.5,0) {};
    \foreach \a in {1,...,4} {
        \draw[thick] (\a + 0.55, \a + 1) -- (\a + 1.45, \a + 1);
        \node [circle, draw, fill, line width = .5pt, inner sep = 0pt, minimum size = 3pt] (ca) at (\a + 0.5, \a + 1) {};
        \node [circle, draw, fill=none, line width = .5pt, inner sep = 0pt, minimum size = 3pt] (ca) at (\a + 1.5, \a + 1) {};
    };
    \foreach \a in {-4,...,-1} {
        \draw[thick] (\a - 0.55, \a - 1) -- (\a - 1.45, \a - 1);
        \node [circle, draw, fill, line width = .5pt, inner sep = 0pt, minimum size = 3pt] (ca) at (\a - 0.5, \a - 1) {};
        \node [circle, draw, fill=none, line width = .5pt, inner sep = 0pt, minimum size = 3pt] (ca) at (\a - 1.5, \a - 1) {};
    };
    \draw[dashed] (-5.5,-5.5) -- (5.5, 5.5);
  \end{tikzpicture}
  \label{Fig1}
  \caption{The graph of $z(y) = IQ_i(DQ_i(y))$, $\Delta_i=10$, $\delta_i=-5$, $\Gamma_i=10$, $\gamma_i=10$}
\end{figure}
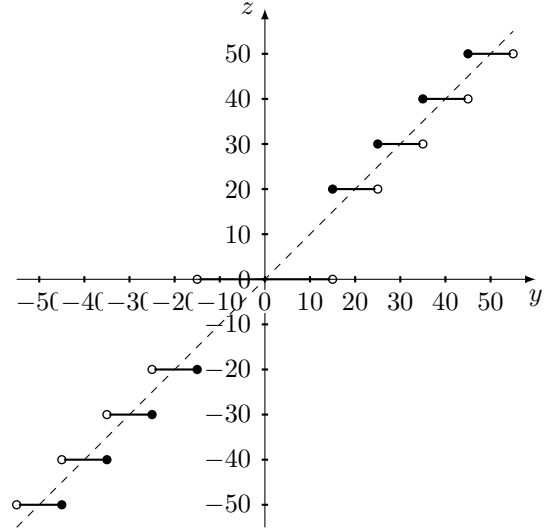

We can write 
$$
  \|\mathcal{E}(\mathbf{y})\|_\infty \leq \max_i \sup_{y \leq \|\mathbf{y}\|_\infty} |IQ_i(DQ_i(y)) - y| .
$$

Let $B_i(y) = \frac{\max(0, |y| + \delta_i)}{\Delta_i}$.
Using the fact that $IQ_i(DQ_i(y)) - y$ is piecewise linear and odd, it is enough to consider the points 
\begin{gather*}
  y_{ij} = j  \Delta_i - \delta_i, \quad1 \leq j \leq \left\lfloor B_i(\|\mathbf{y}\|_\infty)\right\rfloor,
\shortintertext{and} 
  y_{ij} - 0 =  j \Delta_i - \delta_i - 0,\quad 1 \leq j \leq \left\lceil B_i(\|\mathbf{y}\|_\infty)\right\rceil.
\end{gather*}

For these points, 
\begin{gather*}
\begin{multlined}
  IQ_i(DQ_i(y_{ij})) \\ = \!\left\lfloor\frac{j  \Delta_i - \delta_i + \delta_i}{\Delta_i}\right\rfloor\!\Gamma_i + \gamma_i = j  \Gamma + \gamma_i,
\end{multlined}
\\
\begin{multlined}
  IQ_i(DQ_i(y_{ij} - 0)) 
  \\ = 
  \!\left\{
  \begin{array}{ll}
    \!\left\lfloor\frac{j \Delta_i - \delta_i + \delta_i - 0}{\Delta_i}\right\rfloor\!\Gamma_i + \gamma_i & \text{if } j > 1,\\
    0 & \text{if } j = 1
  \end{array}
  \right.\!
\\  =   \left\{
  \begin{array}{ll}
    (j - 1) \Gamma_i + \gamma_i & \text{if } j > 1,\\
    0 & \text{if } j = 1,
  \end{array}
\right.
\end{multlined}
\end{gather*}
where $IQ_i(DQ_i(y_{ij} - 0))$ means the left-sided limit.

Consider the case $\Delta_i = \Gamma_i$.
We have
\begin{gather*}
\begin{multlined}
  IQ_i(DQ_i(y_{ij})) - y_{ij} \\= j \Gamma_i + \gamma_i - j \Delta_i + \delta_i = \gamma_i + \delta_i,
\end{multlined}
\\
\begin{multlined}
  IQ_i(DQ_i(y_{ij} - 0)) - y_{ij}
  \\ = 
  \left\{
  \begin{array}{ll}
    (j - 1)  \Gamma_i + \gamma_i - j  \Delta_i + \delta_i & \text{if } j > 1,\\
    -\Delta_i + \delta_i & \text{if } j = 1
  \end{array}
  \right.\!
 \\ = 
  \left\{
  \begin{array}{ll}
    \gamma_i + \delta_i - \Gamma_i & \text{if } j > 1,\\
    \delta_i - \Delta_i & \text{if } j = 1.
  \end{array}
  \right.
\end{multlined}
\shortintertext{Thus,}
\begin{multlined}
  |IQ_i(DQ_i(y)) - y|
  \\ \leq \max(|\gamma_i + \delta_i|, |\gamma_i + \delta_i - \Gamma_i|, |\Delta_i - \delta_i|).
\end{multlined}
\end{gather*}

Therefore, if $\Delta_i = \Gamma_i$, then
\begin{multline*}
  \|\mathbf{x}' - \mathbf{x}\|_\infty
  \\ 
  \leq n \max_i \max(|\gamma_i + \delta_i|, |\gamma_i + \delta_i - \Delta_i|, |\Delta_i - \delta_i|).
\end{multline*}
Note that this bound does not depend on $\mathbf{x}$.

Now consider the case $\Delta_i \neq \Gamma_i$.
This case is more complicated.
We have
\begin{gather*}
\begin{multlined}
  IQ_i(DQ_i(y_{ij})) - y_{ij} = j \Gamma_i + \gamma_i - j \Delta_i + \delta_i 
  \\ = j (\Gamma_i - \Delta_i) + \gamma_i + \delta_i,
\end{multlined}
\\
\begin{multlined}
  IQ(DQ(y_j - 0)) - y_j
  \\= 
  \left\{
  \begin{array}{ll}
    (j - 1)  \Gamma_i + \gamma_i - j \Delta_i + \delta_i & \text{if } j > 1,\\
    -\Delta_i + \delta_i & \text{if } j = 1
  \end{array}
  \right.\!
\\  = 
  \left\{
  \begin{array}{ll}
    j (\Gamma_i - \Delta_i) + \gamma_i + \delta_i - \Gamma_i & \text{if } j > 1,\\
    \delta_i - \Delta_i & \text{if } j = 1.
  \end{array}
  \right.
\end{multlined}
\shortintertext{Therefore,}
\begin{multlined}
  \|\mathcal{E}(\mathbf{y})\|_\infty 
 \leq
  \max_i
	\max(\{
		|j (\Gamma_i - \Delta_i) + \gamma_i + \delta_i| : \\ 1 \leq j \leq \left\lfloor B_i(\|\mathbf{y}\|_\infty)\right\rfloor
	\} 
     \cup 
    \{|\Delta_i - \delta_i|\} 
    \\ \cup
	\{
		|j  (\Gamma_i - \Delta_i) + \gamma_i + \delta_i - \Gamma_i| : \\ 2 \leq j \leq \left\lceil B_i(\|\mathbf{y}\|_\infty)\right\rceil
	\}).
\end{multlined}
\end{gather*}

Since $U_i(j) = j (\Gamma_i - \Delta_i) + \gamma_i + \delta_i$ and $V_i(j) = j (\Gamma_i - \Delta_i) + \gamma_i + \delta_i - \Gamma_i$ are arithmetic sequences, this expression can be simplified:
\begin{multline*}
  \|\mathcal{E}(\mathbf{y})\|_\infty 
  \leq 
  \max_i
  \max(
   \underbrace{|U_i(1)|, |U_i(\lfloor B_i(\|\mathbf{y}\|_\infty)\rfloor)|}_{\text{if } \lfloor B_i(\|\mathbf{y}\|_\infty)\rfloor \geq 1 },\\
   |\Delta_i - \delta_i|,
   \underbrace{|V_i(2)|, |V_i(\lceil B_i(\|\mathbf{y}\|_\infty)\rceil)|}_{\text{if }\lceil B_i(\|\mathbf{y}\|_\infty)\rceil \geq 2 }
  ).
\end{multline*}

Using the inequality~(\ref{Est1}) and the fact that $\|\mathbf{y}\|_\infty \leq \|\mathbf{x}\|_\infty$, we obtain the following theorem.
\begin{theorem}\label{Theorem1}
The following inequality holds:
  \begin{multline}\label{ErrorBound}
    \|\mathbf{x}' - \mathbf{x}\|_\infty \leq   n   
    \max_i 
  \max\biggl(
   \underbrace{ |(\Gamma_i - \Delta_i) + \gamma_i + \delta_i|  }_{\text{if } \|\mathbf{x}\|_\infty \geq \Delta_i - \delta_i },  \\
    \underbrace{\left|\left\lfloor \frac{\|\mathbf{x}\|_\infty + \delta_i}{\Delta_i} \right\rfloor (\Gamma_i - \Delta_i) + \gamma_i + \delta_i \right|
   }_{\text{if } \|\mathbf{x}\|_\infty \geq \Delta_i - \delta_i }, 
   |\Delta_i - \delta_i|, \\
   \underbrace{
  | 2 (\Gamma_i - \Delta_i) + \gamma_i + \delta_i - \Gamma_i| }_{\text{if }\|\mathbf{x}\|_\infty >  \Delta_i - \delta_i },\\
   \underbrace{
  \left| \left\lceil \frac{\|\mathbf{x}\|_\infty + \delta_i}{\Delta_i} \right\rceil (\Gamma_i - \Delta_i) + \gamma_i + \delta_i - \Gamma_i \right|
   }_{\text{if }\|\mathbf{x}\|_\infty >  \Delta_i - \delta_i }
  \biggr).
  \end{multline}
  If $\Delta_i = \Gamma_i$, then it can be simplified as
  \begin{multline}\label{ErroBound1}
    \|\mathbf{x}' - \mathbf{x}\|_\infty \leq n \max_i \max(|\gamma_i + \delta_i|,\\ |\gamma_i + \delta_i - \Delta_i|, |\Delta_i - \delta_i|).
  \end{multline}
\end{theorem}

\begin{example}
Let $n=16$, $\Delta_i = \Gamma_i = 800$, $\delta_i=-1000$, and $\gamma_i=1400$.
From Theorem~\ref{Theorem1}, it follows that
\begin{multline*}
  \|\mathbf{x}' - \mathbf{x}\|_\infty \leq 16 \cdot \max(|1400 - 1000|,\\ |1400 - 1000 - 800|,  |800 + 1000|) = 28800.
\end{multline*}
\end{example}

Now, consider the asymptotic behavior of the bound~(\ref{ErrorBound}).
If $\|\mathbf{x}\|_\infty$ is large enough and there is $i$ such that $\Delta_i \neq \Gamma_i$, we can leave only the second and fifth arguments of $\max$ in~(\ref{ErrorBound}).
Let ${i^{*}}$ be any index such that $\frac{|\Gamma_{i^{*}} - \Delta_{i^{*}}|}{\Delta_{i^{*}}} =  \max_i
\frac{|\Gamma_{i} - \Delta_{i}|}{\Delta_i}$.
Then we have the following inequality for the relative error:
\begin{equation*}
  \frac{\|\mathbf{x}' - \mathbf{x}\|_\infty }{\|\mathbf{x}\|_\infty } 
  \leq
  n \frac{|\Gamma_{i^{*}} - \Delta_{i^{*}}|}{\Delta_{i^{*}}} + \varepsilon(\|\mathbf{x}\|_\infty),
\end{equation*}
where $\varepsilon(\|\mathbf{x}\|_\infty) = o(\|\mathbf{x}\|_\infty)$ as $x$ tends to $\infty$.

\section{Bounds on the largest absolute value}\label{LargestSection}

In this section, we discuss the largest absolute value that the components of $\mathbf{x}' = \mathbf{IT}(\mathbf{IQ}(\mathbf{DQ}(\mathbf{DT}(\mathbf{x}))))$ can attain.
To obtain a bound on this value, we can, of course, use the error bounds~(\ref{ErrorBound}) and~(\ref{ErroBound1}) from Theorem~\ref{Theorem1}.
For example, from~(\ref{ErroBound1}) we have:
\begin{multline}\label{CaseWO1}
  \|\mathbf{x}'\|_\infty \leq \| \mathbf{x}\|_\infty +
   n\max_i \max(|\gamma_i + \delta_i|,\\ |\gamma_i + \delta_i - \Delta_i|, |\Delta_i - \delta_i|).
  \end{multline}

\begin{example}\label{LargestSectionEx1}
Let $\|\mathbf{x}\|_\infty \leq 2048$, $n=16$, $\Delta_i = \Gamma_i = 800$, $\delta_i=-1000$, and $\gamma_i=1400$.
The formula~(\ref{CaseWO1}) gives us
\begin{multline*}
  \|\mathbf{x}'\|_\infty \leq 2048 + 16 \cdot \max(|1400 - 1000|, \\ |1400 - 1000 - 800|, |800 + 1000|) = 30848.
\end{multline*}
\end{example}
We can use another approach to find a bound on $\|\mathbf{x}'\|_\infty$.
Let $\mathbf{y} = \frac{1}{n}\mathbf{H}(\mathbf{x})$ and $\mathbf{z} = \mathbf{IQ}(\mathbf{DQ}(\mathbf{y}))$.
Using Example~\ref{Ex1}, it is easy to see that
\begin{gather}\nonumber
\|\mathbf{x}'\|_\infty = \|\mathbf{IT}(\mathbf{z})\|_\infty \leq
\|\mathbf{H}^T\|_\infty  \|\mathbf{z}\|_\infty = n  \|\mathbf{z}\|_\infty.
\shortintertext{Note that} 
\label{ZNorm}
\|\mathbf{z}\|_\infty \leq 
\max_i IQ_i(DQ_i(\|\mathbf{x}\|_\infty)).
\shortintertext{Therefore,}
\nonumber
  \|\mathbf{x}'\|_\infty \leq 
n  \max_i IQ_i(DQ_i(\|\mathbf{x}\|_\infty)).
\end{gather}

To improve this bound, we can use the $\|\cdot\|_{1, \infty}$ and
 $\|\cdot\|_{1}$ norms.
Indeed, note that
\begin{multline*}
\|\mathbf{z}\|_1 = |z_1| + \cdots + |z_n|
\\ \leq 
\sum_{|z_i| = 0} |z_i| + \sum_{|z_i| \neq 0} |z_i| \leq K  \|\mathbf{z}\|_\infty,
\end{multline*}
where $K$ is the number of the non-zero components of $\mathbf{z}$.
Hence, 
\begin{multline*}
  \|\mathbf{x}'\|_\infty = \|\mathbf{IT}(\mathbf{z})\|_\infty \leq \|\mathbf{IT}\|_{1, \infty}  \|\mathbf{z}\|_1 = 1  \|\mathbf{z}\|_1 \\ = K  \|\mathbf{z}\|_\infty.
\end{multline*}
Using~(\ref{ZNorm}), we obtain
\begin{equation}\label{K2}
  \|\mathbf{x}'\|_\infty \leq  K  
  \max_i IQ_i(DQ_i(\|\mathbf{x}\|_\infty)).
\end{equation} 
Therefore, we need to find a bound on the number of nonzero components $K$.

Note that 
\begin{multline*}
  \|\mathbf{y}\|_1 = |y_1| + \ldots + |y_n|\\ = 
  \sum_{|y_i| < \Delta_i - \delta_i} |y_i|  + 
  \sum_{|y_i| \geq \Delta_i - \delta_i} |y_i|  
  \geq 
  \sum_{|y_i| \geq \Delta_i - \delta_i} |y_i|  \\ \geq
\sum_{|y_i| \geq \Delta_i - \delta_i} \min_i(\Delta_i - \delta_i) 
=
K \min_i(\Delta_i - \delta_i).
\end{multline*}

We know that $\|\mathbf{H}\mathbf{x}\|_1 \leq \|\mathbf{H}\|_{\infty, 1}  \|\mathbf{x}\|_\infty$.
Hence, using~(\ref{Formula2}), we have $$\|\mathbf{y}\|_1 \leq \frac{1}{n}\|\mathbf{H}\|_{\infty, 1}\|\mathbf{x}\|_\infty \leq \frac{n^{3/2}}{n}  \|\mathbf{x}\|_\infty = \sqrt{n}  \|\mathbf{x}\|_\infty .$$
Therefore, 
\begin{equation}\label{K1}
  K \leq \min\!\left(n, \left\lfloor\frac{\sqrt{n}  \|\mathbf{x}\|_\infty}{\min_i(\Delta_i - \delta_i)}\right\rfloor\right)\!.
\end{equation}

Combining the inequalities~(\ref{K1}) and~(\ref{K2}), we obtain the following result.

\begin{theorem}\label{Theorem2} The following inequality holds:
  \begin{multline}\label{IneqSec2}
    \|\mathbf{x}'\|_\infty \leq \min\!\left(n, \left\lfloor\frac{\sqrt{n} \|\mathbf{x}\|_\infty}{\min_i(\Delta_i - \delta_i)}\right\rfloor \right) \\ \cdot \max_i IQ_i(DQ_i(\|\mathbf{x}\|_\infty)).
  \end{multline}
\end{theorem}
This theorem should be used in combination with Theorem~\ref{Theorem1}.
Given $\Delta_i$, $\Gamma_i$, $\delta_i$, $\gamma_i$, and the bound on $\|\mathbf{x}\|_\infty$, we need to calculate the bounds based on Theorem~\ref{Theorem1} and Theorem~\ref{Theorem2} and choose the best result.

\begin{example}
  Consider the parameters given in Example~\ref{LargestSectionEx1}: $\|\mathbf{x}\|_\infty \leq 2048$, $n=16$, $\Delta_i = \Gamma_i = 800$, $\delta_i=-1000$, and $\gamma_i=1400$.
  In this case, the formula~(\ref{IneqSec2}) gives us the following bound
\begin{multline*}
  \|\mathbf{x}'\|_\infty \leq \left\lfloor\frac{\sqrt{16} \cdot 2048}{800 + 1000}\right\rfloor  \\\cdot \left(\left\lfloor\frac{2048 - 1000}{800}\right\rfloor\! \cdot 800 + 1400\right) = 8800.
\end{multline*}
But the formula~(\ref{CaseWO1}) gives us a worse result:
$$
  \|\mathbf{x}'\|_\infty \leq 30848.
$$
\end{example}

\section{Relation to the excess of a matrix}\label{ConnectionSection}

As we saw in the previous section, the norm $\|\mathbf{H}\|_{\infty, 1}$ was useful to find a bound on $\|\mathbf{x}\|_\infty$. It would be interesting to know the exact value of this norm.
Unfortunately, Example~\ref{Ex2} gives us an upper bound only and the exact values only for small matrices and for orders $2^k$, where $k$ is even.
In this section, we consider the connection between the norm $\|\cdot\|_{\infty, 1}$ and the excess of a matrix.
The excess of a Hadamard matrix has been widely studied for many years~\cite{SchmidtWang1977,Best1977,EnomotoMiyamoto1980,KouniasFarmakis1988,HirasakaMomiharaSuda2017}.
The problem of finding the exact value is combinatorial, and fast methods are unknown, but many lower and upper bounds were obtained.
The main result of this section is Theorem~\ref{Theorem3}.

For $\mathbf{H} \in \mathcal{H}_n$, the sum of all the entries of $\mathbf{H}$ is called the \textit{excess} of $\mathbf{H}$ and is denoted by $\sigma(\mathbf{H})$.
For a subset $\mathcal{S}$ of $\mathcal{H}_n$, the \textit{maximal excess} of $\mathcal{S}$ is defined as 
$$
\sigma(\mathcal{S}) = \max\{\sigma(\mathbf{H}) : \mathbf{H} \in \mathcal{S}\}.
$$
Two Hadamard matrices are \textit{equivalent} if one of them can be transformed to the other by permutation and negation of rows and columns. The equivalence class containing $\mathbf{H}$ is denoted by $[\mathbf{H}]$.

The exact values of the maximal excess are known only for small Hadamard matrices. 
Schmidt and Wang~\cite{SchmidtWang1977}, Best~\cite{Best1977}, Enomoto and Miyamoto~\cite{EnomotoMiyamoto1980}, and other authors obtained some upper and lower bounds on $\sigma(\mathbf{H})$. Let us recall Best's result.

\begin{theorem}[Best \cite{Best1977}]
For any $\mathbf{H} \in \mathcal{H}_n$, the following inequalities hold:
\begin{gather*}
\sigma([\mathbf{H}]) \geq \frac{n^2}{2^n}\binom{n}{n/2}, \quad \sigma([\mathbf{H}]) \leq n^\frac{3}{2},
\\
\shortintertext{and}
\begin{multlined}
\sigma([\mathbf{H}]) \leq  \max\Biggl\{ \sum_{i = 1}^n s_i :   \sum_{i = 1}^n s_i^2 = n^2,  s_i \in 2\mathbb{Z}, \\ s_i \equiv s_j \bmod 4\Biggr\},
\end{multlined}
\end{gather*}
\end{theorem}

Enomoto and Miyamoto proved the following two theorems. We recall them to demonstrate that better bounds are quite cumbersome.
\begin{theorem}[Enomoto and Miyamoto \cite{EnomotoMiyamoto1980}]  For any $\mathbf{H} \in \mathcal{H}_n$, the following inequality holds:
\begin{gather*}
\sigma([\mathbf{H}]) \geq \max_{m \in \mathbb{N} \,:\, 1 \leq m \leq n}{Q_1(n, m)},
\shortintertext{where}
Q_1(n, m) =|n - 2m| + \frac{2(n - 1) P_1(n, m)}{\binom{n}{m}}
\shortintertext{and}
P_1(n, m) = 
\left\{
\begin{array}{ll}
n\binom{\frac{n}{2} - 1}{\frac{m - 1}{2}}^2 & \text{ if } m \in 2\mathbb{N} + 1, \\
m\binom{\frac{n}{2}}{\frac{m}{2}}^2 - n \binom{\frac{n}{2} - 1}{\frac{m}{2} - 1}^2 & \text{ if } m \in 2\mathbb{N}.
\end{array}
\right.
\end{gather*}
\end{theorem}
\begin{theorem}[Enomoto and Miyamoto \cite{EnomotoMiyamoto1980}] For any $\mathbf{H} \in \mathcal{H}_n$, the following inequality holds:
\begin{gather*}
\sigma([\mathbf{H}]) \geq \max_{m \in 2\mathbb{N} + 1 \,:\, \frac{n}{4} \leq m \leq \frac{3n}{4}}{Q_2(n, m)},
\shortintertext{where}
Q_2(n, m) = |2n - 4m|  + \frac{n(n - 2)P_2(n, m)}{\binom{\frac{n}{2}}{\frac{n}{4}} \binom{\frac{n}{2}}{m - \frac{n}{4}}}
\shortintertext{and}
\begin{multlined}
P_2(n, m) = \sum_{j = 0}^{m - \frac{n}{4}}{\!\binom{\frac{n}{4} - 1}{\frac{m - 1}{2} - j}\!}^2 \!\binom{\frac{n}{4}}{j}\!\binom{\frac{n}{4}}{m - \frac{n}{4} - j} + \\ 
\sum_{j = 0}^{m - \frac{n}{4} - 1}{\!\binom{\frac{n}{4}}{\frac{m - 1}{2} - j}\!}^2 \!\binom{\frac{n}{4} - 1}{j}\!\binom{\frac{n}{4} - 1}{m - \frac{n}{4} - j - 1}.    
\end{multlined}
\end{gather*}
\end{theorem}

Since $\sigma(\mathbf{H}_1 \otimes \mathbf{H}_2) = \sigma(\mathbf{H}_1) \sigma(\mathbf{H}_2)$
and $[\mathbf{H}_1 \otimes \mathbf{H}_2] \supseteq [\mathbf{H}_1] \otimes [\mathbf{H}_2]$,
it follows that
$\sigma([\mathbf{H}_1 \otimes \mathbf{H}_2]) \geq \sigma([\mathbf{H}_1])  \sigma([\mathbf{H}_2])$. This fact is a generalization of Example~\ref{Ex3}.

The next theorem establishes the connection between $\sigma([\mathbf{H}])$ and $\|\mathbf{H}\|_{\infty, 1}$.

\begin{theorem}\label{Theorem3}
  For any $\mathbf{H} \in \mathcal{H}_n$, the following equality holds:
  $$
    \sigma([\mathbf{H}]) = \|\mathbf{H}\|_{\infty, 1}.
  $$\end{theorem}
\begin{proof}
If two Hadamard matrices $\mathbf{H}_1$ and $\mathbf{H}_2$ are equivalent, then $\mathbf{H}_2 = \mathbf{Q}_1^{-1} \mathbf{H}_1 \mathbf{Q}_2,$ where $\mathbf{Q}_1$ and $\mathbf{Q}_2$ are monomial matrices (having only one nonzero element in each row or column) with nonzero entries $\pm1$.
 Monomial matrices can be presented as a product of a diagonal matrix and a permutation matrix: $\mathbf{Q} = \mathbf{D} \mathbf{P}$.
Therefore, 
\begin{multline*}
[\mathbf{H}] = \{\mathbf{P}_1 \mathbf{D}_1 \mathbf{H} \mathbf{D}_2 \mathbf{P}_2 \,:\, \mathbf{P}_1, \mathbf{P}_2 \in \mathrm{Perm}(n), \\ \mathbf{D}_1, \mathbf{D}_2 \in \mathrm{Diag}(n, \{-1, 1\})\}.
\end{multline*}

Since permutations of rows and columns of a matrix $\mathbf{H}$ do not change $\sigma(\mathbf{H})$, it follows that
$$\sigma([\mathbf{H}]) = \sigma( \{\mathbf{D}_1 \mathbf{H} \mathbf{D}_2 \,:\, \mathbf{D}_1, \mathbf{D}_2 \in \mathrm{Diag}(n, \{-1, 1\})\}).$$

Note that 
\begin{multline*}
\sigma(\mathbf{D}_1 \mathbf{H} \mathbf{D}_2) = \sum_{i = 1}^n\sum_{j = 1}^n  {d}_{1ii} {h}_{ij} {d}_{2jj}  \\= \sum_{i = 1}^n {d}_{1ii}  \sum_{j = 1}^n  {h}_{ij} {d}_{2jj}.
\end{multline*}

Let $s(x) = 
\left\{
\begin{array}{ll}
1 & \text{ if } x \geq 0, \\
-1 &  \text{ if } x < 0,
\end{array}
\right.$, $y_i = {d}_{1ii}$, and $x_j = {d}_{2jj}$.
We have
\begin{multline*}
\sigma([\mathbf{H}]) 
= 
\max_{x_j, y_i \in \{-1, 1\}} \sum_{i = 1}^n y_{i}  \sum_{j = 1}^n h_{ij} x_{j}  
\\=
\max_{x_j \in \{-1, 1\}}\sum_{i = 1}^n s\!\left(\sum_{j = 1}^n h_{ij} x_{j}\right)  \sum_{j = 1}^n h_{ij} x_{j} 
 \\=
\max_{x_j \in \{-1, 1\}} \sum_{i = 1}^n \left|\sum_{j = 1}^n  h_{ij} x_{j}\right|\!.
\end{multline*}

On the other hand,
from Formula~(\ref{NormInfty1}), we have
\begin{multline*}
\|\mathbf{H}\|_{\infty, 1} = \max\{\|\mathbf{H}\mathbf{x}\|_1 : \mathbf{x} \in \{-1, 1\}^n\}  \\=
\max\{|\mathbf{h}_{1}\mathbf{x}| + |\mathbf{h}_{2}\mathbf{x}| + \ldots + |\mathbf{h}_{n}\mathbf{x}|: \mathbf{x} \in \{-1, 1\}^n\}  \\=
\max_{\mathbf{x} \in \{-1, 1\}^n } \sum_{i = 1}^n \left|\sum_{j = 1}^n  h_{ij} x_{j}\right|\!.
\end{multline*}
As we can see, we obtain the same expression.
\end{proof}

\section{Conclusion}
In this paper, we obtained the upper bounds on the error and the largest absolute value.
Our formulae are helpful for analyzing the impact of the quantization error and allow us to find a bound on the number of bits we need to store results.

Our methods can also be used for other quantization and inverse quantization functions and for other pipelines of the form $\mathbf{x}' = \mathbf{L}_2 (\mathbf{N} (\mathbf{L}_1 (\mathbf{x})))$, where $\mathbf{L}_1$ and $\mathbf{L}_2$ are linear transformations and $\mathbf{N}$ is a nonlinear transformation.

In addition, we demonstrated the connection between the norm $\|\mathbf{H}\|_{\infty, 1}$ of a Hadamard matrix $\mathbf{H}$ and the maximal excess $\sigma([\mathbf{H}])$ of the equivalence class containing $\mathbf{H}$. 

It is known that computing the norm $\|\mathbf{A}\|_{\infty, 1}$ for a matrix is $\mathsf{NP}$-hard~\cite{Rohn2000}. Also, fast methods for computing the excess of a Hadamard matrix are still unknown. Hence, it is natural to ask the following question about Hadamard matrices:
\begin{question}
Is computing the norm $\|\mathbf{H}\|_{\infty, 1}$ for Hadamard matrices $\mathsf{NP}$-hard?
\end{question}

Matrices constructed by Sylvester's construction is a very special class,
so maybe computing this norm for this class is easier. For this class, we would like to ask the following question: 
\begin{question}
Find a formula to calculate $\|\mathbf{H}_{2^k}\|_{\infty, 1}$.
\end{question}

{\small
\bibliographystyle{amsplain}
\bibliography{report-001}
}
\end{document}